\documentclass[11pt]{amsart}
\usepackage{mathrsfs,eucal,amsmath,amssymb,latexsym,dsfont,stmaryrd,enumerate}
\usepackage{tikz-cd}
\usepackage{tikz}
\usepackage{enumerate}
\usepackage[all,2cell]{xy} \UseAllTwocells \SilentMatrices
\usepackage[T1]{fontenc}
\usepackage{xcolor}
\usepackage{hyperref}
\begin{document}

\newcommand{\INVISIBLE}[1]{}

\newtheorem{thm}{Theorem}[section]
\newtheorem{lem}[thm]{Lemma}
\newtheorem{cor}[thm]{Corollary}
\newtheorem{prp}[thm]{Proposition}
\newtheorem{conj}[thm]{Conjecture}

\theoremstyle{definition}
\newtheorem{dfn}[thm]{Definition}
\newtheorem{question}[thm]{Question}
\newtheorem{nota}[thm]{Notations}
\newtheorem{notation}[thm]{Notation}
\newtheorem*{claim*}{Claim}
\newtheorem{ex}[thm]{Example}
\newtheorem{counterex}[thm]{Counter-example}
\newtheorem{rmk}[thm]{Remark}
\newtheorem{rmks}[thm]{Remarks}
 
\def\labelenumi{(\arabic{enumi})}

\newcommand{\aro}{\longrightarrow}
\newcommand{\arou}[1]{\stackrel{#1}{\longrightarrow}}
\newcommand{\RA}{\Longrightarrow}

\newcommand{\mm}[1]{\mathrm{#1}}
\newcommand{\bm}[1]{\boldsymbol{#1}}
\newcommand{\bb}[1]{\mathbf{#1}}

\newcommand{\bA}{\boldsymbol A}
\newcommand{\bB}{\boldsymbol B}
\newcommand{\bC}{\boldsymbol C}
\newcommand{\bD}{\boldsymbol D}
\newcommand{\bE}{\boldsymbol E}
\newcommand{\bF}{\boldsymbol F}
\newcommand{\bG}{\boldsymbol G}
\newcommand{\bH}{\boldsymbol H}
\newcommand{\bI}{\boldsymbol I}
\newcommand{\bJ}{\boldsymbol J}
\newcommand{\bK}{\boldsymbol K}
\newcommand{\bL}{\boldsymbol L}
\newcommand{\bM}{\boldsymbol M}
\newcommand{\bN}{\boldsymbol N}
\newcommand{\bO}{\boldsymbol O}
\newcommand{\bP}{\boldsymbol P}
\newcommand{\bY}{\boldsymbol Y}
\newcommand{\bS}{\boldsymbol S}
\newcommand{\bX}{\boldsymbol X}
\newcommand{\bZ}{\boldsymbol Z}

\newcommand{\cc}[1]{\mathcal{#1}}
\newcommand{\ccc}[1]{\mathscr{#1}}

\newcommand{\ca}{\cc{A}}

\newcommand{\cb}{\cc{B}}

\newcommand{\cC}{\cc{C}}

\newcommand{\cd}{\cc{D}}

\newcommand{\ce}{\cc{E}}

\newcommand{\cf}{\cc{F}}

\newcommand{\cg}{\cc{G}}

\newcommand{\ch}{\cc{H}}

\newcommand{\ci}{\cc{I}}

\newcommand{\cj}{\cc{J}}

\newcommand{\ck}{\cc{K}}

\newcommand{\cl}{\cc{L}}

\newcommand{\cm}{\cc{M}}

\newcommand{\cn}{\cc{N}}

\newcommand{\co}{\cc{O}}

\newcommand{\cp}{\cc{P}}

\newcommand{\cq}{\cc{Q}}

\newcommand{\cR}{\cc{R}}

\newcommand{\cs}{\cc{S}}

\newcommand{\ct}{\cc{T}}

\newcommand{\cu}{\cc{U}}

\newcommand{\cv}{\cc{V}}

\newcommand{\cy}{\cc{Y}}

\newcommand{\cw}{\cc{W}}

\newcommand{\cz}{\cc{Z}}

\newcommand{\cx}{\cc{X}}

\newcommand{\g}[1]{\mathfrak{#1}}

\newcommand{\af}{\mathds{A}}
\newcommand{\PP}{\mathds{P}}

\newcommand{\GL}{\mathrm{GL}}
\newcommand{\PGL}{\mathrm{PGL}}
\newcommand{\SL}{\mathrm{SL}}
\newcommand{\NN}{\mathds{N}}
\newcommand{\ZZ}{\mathds{Z}}
\newcommand{\CC}{\mathds{C}}
\newcommand{\QQ}{\mathds{Q}}
\newcommand{\RR}{\mathds{R}}
\newcommand{\FF}{\mathds{F}}
\newcommand{\DD}{\mathds{D}}
\newcommand{\VV}{\mathds{V}}
\newcommand{\HH}{\mathds{H}}
\newcommand{\MM}{\mathds{M}}
\newcommand{\OO}{\mathds{O}}
\newcommand{\LL}{\mathds L}
\newcommand{\BB}{\mathds B}
\newcommand{\kk}{\mathds k}
\newcommand{\bs}{\mathbf S}
\newcommand{\GG}{\mathds G}
\newcommand{\XX}{\mathds X}

\newcommand{\WW}{\mathds W}
\newcommand{\al}{\alpha}

\newcommand{\be}{\beta}

\newcommand{\ga}{\gamma}
\newcommand{\Ga}{\Gamma}

\newcommand{\om}{\omega}
\newcommand{\Om}{\Omega}

\newcommand{\vt}{\vartheta}
\newcommand{\te}{\theta}
\newcommand{\Te}{\Theta}

\newcommand{\ph}{\varphi}
\newcommand{\Ph}{\Phi}

\newcommand{\ps}{\psi}
\newcommand{\Ps}{\Psi}

\newcommand{\ep}{\varepsilon}

\newcommand{\vr}{\varrho}

\newcommand{\de}{\delta}
\newcommand{\De}{\Delta}

\newcommand{\la}{\lambda}
\newcommand{\La}{\Lambda}

\newcommand{\ka}{\kappa}

\newcommand{\si}{\sigma}
\newcommand{\Si}{\Sigma}

\newcommand{\ze}{\zeta}

\newcommand{\fr}[2]{\frac{#1}{#2}}
\newcommand{\vs}{\vspace{0.3cm}}
\newcommand{\na}{\nabla}
\newcommand{\pd}{\partial}
\newcommand{\po}{\cdot}
\newcommand{\met}[2]{\left\langle #1, #2 \right\rangle}
\newcommand{\rep}[2]{\mathrm{Rep}_{#1}(#2)}
\newcommand{\repo}[2]{\mathrm{Rep}^\circ_{#1}(#2)}
\newcommand{\hh}[3]{\mathrm{Hom}_{#1}(#2,#3)}
\newcommand{\modules}[1]{#1\text{-}\mathbf{mod}}
\newcommand{\vect}[1]{#1\text{-}\mathbf{vect}}

\newcommand{\Modules}[1]{#1\text{-}\mathbf{Mod}}
\newcommand{\dmod}[1]{\mathcal{D}(#1)\text{-}{\bf mod}}
\newcommand{\spc}{\mathrm{Spec}\,}
\newcommand{\an}{\mathrm{an}}
\newcommand{\NNo}{\NN\smallsetminus\{0\}}

\newcommand{\pos}[2]{#1\llbracket#2\rrbracket}

\newcommand{\lau}[2]{#1(\!(#2)\!)}

\newcommand{\cpos}[2]{#1\langle#2\rangle}

\newcommand{\id}{\mathrm{id}}

\newcommand{\one}{\mathds 1}

\newcommand{\ti}{\times}
\newcommand{\tiu}[1]{\underset{#1}{\times}}

\newcommand{\ot}{\otimes}
\newcommand{\otu}[1]{\underset{#1}{\otimes}}

\newcommand{\wh}{\widehat}
\newcommand{\wt}{\widetilde}
\newcommand{\ov}[1]{\overline{#1}}
\newcommand{\un}[1]{\underline{#1}}

\newcommand{\op}{\oplus}

\newcommand{\lid}{\varinjlim}
\newcommand{\lip}{\varprojlim}

\newcommand{\mcrs}{\bb{MC}_{\rm rs}}
\newcommand{\mclog}{\bb{MC}_{\rm log}}
\newcommand{\mc}{\mathbf{MC}}

\newcommand{\ega}[3]{[EGA $\mathrm{#1}_{\mathrm{#2}}$, #3]}

\newcommand{\asts}{\begin{center}$***$\end{center}}
\title[Prolongation of regular-singular connections]{Prolongation of regular-singular connections on punctured affine line over a Henselian ring}

\author[Hai]{Ph\`ung H\^o Hai}
\address[Hai]{Institute of Mathematics, Vietnam Academy of Science and Technology, 18 Hoang Quoc Viet, Cau Giay, Hanoi}
\email{phung@math.ac.vn}

\author[dos Santos]{Jo\~ao Pedro dos Santos}
\address[dos Santos]{Institut Montpelli\'erain A. Grothendieck,  
Case courrier 051,
Place Eug\`ene Bataillon,
34090 Montpellier, France }
\email{joao\_pedro.dos\_santos@yahoo.com}

\author[Tam]{Ph\d am Thanh T\^am}
\address[T\^ am]{Department of Mathematics, Hanoi Pedagogical University 2, Vinh Phuc, Vietnam}   
\email{phamthanhtam@hpu2.edu.vn}

\author[Thinh]{\DJ\`ao V\u{a}n Th\d inh}
\address[Thinh]{Institute of Mathematics, Vietnam Academy of Science and Technology, 18 Hoang Quoc Viet, Cau Giay, Hanoi}
\email{daothinh1812@gmail.com}


\begin{abstract}
We generalize Deligne's equivalence between the categories of regular-singular connections on the formal punctured disk and on the punctured affine line   to the case where the base is a  strictly Henselian discrete valuation ring  of equal characteristic 0. We also provide a weaker result when the base is higher dimensional.   
\end{abstract}
\thanks{The research of Ph\`ung H\^o Hai and Ph\d am Thanh T\^am is funded by  the International Center for Research and Postgraduate Training in Mathematics (Institute of Mathematics, VAST, Vietnam) under	grant number  \texttt{ICRTM01\_2020.06} and funded by Vingroup Joint Stock Company and supported by Vingroup Innovation Foundation (VinIF) under the project code VINIF.2021.DA00030.
}
\thanks{The research of \DJ\`ao V\u an Th\d inh was supported by the Postdoctoral program of Vietnam Institute for Advanced Study in Mathematics
}
\date{\today}
\maketitle

\section{Introduction}
Let $C$ be an algebraically closed field of characteristic 0 and $x$ be a variable.  The formal punctured disk,  ${\rm Spec}\,\lau Cx$,  is equipped with the ``vector field'' $\vt:=x\dfrac{d}{dx}$. In \cite[Proposition~13.35]{Del87}, Deligne established an  equivalence between regular-singular connections on the formal punctured disk  and on the punctured affine line $\PP_C^1\smallsetminus \{0,\infty\}$:
\[
\left\{\begin{array}{c}\text{Regular-singular connections}\\\text{on the punctured formal disk}\end{array}\right\}\simeq\left\{\begin{array}{c}\text{Regular-singular connections}\\\text{on the punctured affine line}\end{array}\right\}.
\] 
 Deligne uses this equivalence to produce ``fiber functors'' from the category of regular-singular connections on the formal punctured disk, the   tangential fiber functors. 

Deligne's equivalence was also considered by Katz in a more general setting \cite{Kat87}. The analogs in characteristic $p$ were essentially established by Gieseker in \cite{Gie75} and further developed  by Kindler in \cite{Kin15}. There is also a generalization to the $p$-adic setting by Matsuda \cite{Mat02}, see also \cite{And02}. 

If we now replace $C$ by a complete local noetherian $C$--algebra $R$, Deligne's equivalence possesses a clear analogue, which was proved in \cite[Theorem 10.1]{HdST22}. 
The main idea behind the proof of this result is to make use of the fact that $R$ is a limit of finite dimensional local $C$-algebras and then, turning attention to objects with ``truncated actions'' of $R$, to ``pass to the limit.'' To be more specific, let $\g r$ stand for the maximal ideal of $R$, so that $R_k:=R/\g r^{k+1}$ is a finite $C$-algebra. Now, given a $C$-linear category $\g C$ (such as the category of connections on the punctured affine line), we restrict our attention to the category $\g C_{(R_k)}$ of objects in $\g C$ which have an action of $R_k$. If $\g C'$ is 
another $C$-linear category (e.g. the category of connections on the formal punctured disk)
and if we are able to obtain compatible equivalences $\g C_{(R_k)}\simeq\g C_{(R_k)}'$, it is to be hoped that ``passing to the limit'' will give us a menas to produce equivalences of $R$-linear categories. Clearly, this idea relies heavily on the completeness of $R$. 


In this manuscript, we   deal with the case where $R$ is a noetherian Henselian local $C$--algebra. Our main results are   Theorem \ref{thm.20220623--01} and Theorem \ref{cor.deligne.equivalence.hensel}. In a nutshell:
\begin{thm}\label{main_theorem}
Let $\mcrs(*/R)$ denote the category of regular-singular connections on $*$  over $R$, and 
$\mcrs^{{\circ}}(*/R)$ denote the full subcategory of objects whose   underlying modules are $R$-flat.
Then, the restriction functor 
\[
\bb r: \mcrs^{{\circ}}(R[x^\pm]/R)\aro  \mcrs^{{\circ}}(\lau Rx/R)
\] 
is an equivalence provided that $R$ is a G-ring. If $R$ is, moreover, a discrete valuation ring, then   
\[
\bb r: \mcrs(R[x^\pm]/R)\aro  
\mcrs(\lau Rx/R)
\] 
is  an equivalence. 
\end{thm}

The relevance  of this result is twofold. On the one hand, according to Deligne's point of view \cite{Del87},
it produces  fiber functors for the category $\mcrs(\lau Rx/R)$: we compose
the aforementioned equivalence   with a fiber functor at an $R$-point of $\mm{Spec}\, R[x^\pm]$. 
(Deligne calls these  ``tangential'' fiber  functors.) On the other hand, this 
equivalence describes the structure of $\mcrs(R[x^\pm]/R)$ in terms of 
$\mcrs(\lau Rx/R)$, which is easier to grasp.
Finally, the reduction from a complete noetherian local ring to a noetherian 
Henselian local ring should be an important step toward the case of an arbitrary noetherian local ring. (It is perhaps useful to observe that the class of G-rings is a broad and reasonable one. More on this will be found in the body of the text.)

Our approach is based on 
Deligne's equivalence as presented in   \cite[Theorem 10.1]{HdST22} and Popescu approximation. While the proof of 
\cite[Theorem 10.1]{HdST22} relied on the accessory category of representations of the group $\ZZ$, we have found no reasonable way to include this actor in the present picture; instead, we have made use of its ``Lie version'', which is the category of endomorphisms of $R$-modules. The part concerning Popescu approximation is of course important, but its employment is more straightforward.

The paper is organized as follows.  
 Section \ref{sect2} is devoted to the category of regular-singular connections on the formal relative punctured disk. 
We show that each  connection on a flat $R$-module admits an Euler form. Section \ref{sect.4} is devoted to the category of regular-singular connections on the punctured relative affine line. Similarly, we show that a connection on an $R$-flat module admits an Euler form. The results obtained in these two sections are then used to prove Theorem \ref{main_theorem} in Section \ref{sect.5}.

\subsection{Notation and conventions}\label{NC}
\begin{enumerate}
\item 
[$C$] is a fixed algebraically closed field of characteristic $0$.
\item[$R$] is an integral noetherian local   $C$-algebra with maximal ideal $\frak r$ and residue field isomorphic to $C$.
\item [$R_k$] is the truncation $R/\g r^{k+1}$. 
\item[$\wh R$] is the $\mathfrak r$--adic completion of $R$.  
\item[$\lau Rx$] denotes the ring of formal Laurent series with coefficients in $R$: we have   $\lau Rx=\pos Rx[x^{-1}]$.
\item[$(0:\g a)_M$] is the submodule of all $m\in M$ annihilated by the ideal $\g a$. If $\g a=(a)$, we shall abbreviate $(0:\g a)_M$ to $(0:a)_M$.  
\item[$\vartheta$] denotes  $R$-linear derivation on  $\lau Rx$ given by
$$\vt\sum a_nx^n=x\frac{d}{dx} \sum a_nx^n=\sum na_nx^n.$$

\item[$\mm{Sp}_\ph$] denotes the spectrum of the endomorphism $\ph:V\to V$  of  vector space over $C$.

\item[$\tau$] denotes a subset  of $C$ 
such that the natural map $\tau\to C/\ZZ$ is bijective. In some cases, we shall assume that $0\in\tau$.  
\item[$\bb{End}_R$] denotes  the category of couples $(V,A)$, consisting of a finite $R$-module $V$ and an $R$-linear endomorphism $A:V\to V$;  {\it arrows} from $(V,A)$ to $(V',A')$ are $R$-linear morphisms $\ph:V\to V'$ such that $A'\ph=\ph A$. See also \cite{HdST22}. 
\item[ $\bb{End}_R^{{\circ}}$] denotes  the full subcategory of $\bb{End}_R$ whose objects are flat (and hence free) $R-$modules.  
\end{enumerate}

\section{regular-singular connections on a punctured formal disk}\label{sect.3}
\label{sect2}

\subsection{Definitions and properties} 
We review in this subsection the definitions and main properties of regular-singular connections on the relative formal punctured disk $\mm{Spec}\,\lau Rx$. Our reference is \cite{HdST22}. We notice that although in op. cit. the ring $R$ is assumed to be complete, many results hold in more generality.

\begin{dfn}
[Connections on the punctured formal disk]
\label{dfn.20220607--04}The {\it category of connections} on the punctured formal disk over $R$, or  on $\lau Rx$ over $R$,  or on $\lau Rx/R$,  denoted $\mc(\lau Rx/R)$, has for    
\begin{enumerate}\item[{\it objects}]   those couples $(M,\na)$ consisting of a finite $\lau Rx$-module $M$ and an $R$-linear endomorphism $\na:M\to M$, called  {\it the derivation}, satisfying Leibniz's rule $\na(fm)=\vt(f)m+f\na(m)$, and the  
\item[{\it arrows}] from $(M,\na)$ to  $(M',\na')$ are  $\lau Rx$-linear morphisms $\ph:M\to M'$ such that $\na'\ph=\ph\na$.  
\end{enumerate}
\end{dfn}

The $R$--flat connections on $\lau Rx/R$ enjoy the following remarkable property which is employed further ahead. 
 
\begin{prp}[{\cite[Theorem 8.18]{HdST22}}]
\label{rmk.20220607--01}
	 Let $(M,\na)$ be a connection  on $\lau Rx$ over $R$ such that $M$ is $R$--flat. Then, $M$  is a flat $\lau Rx$--module.
\end{prp}

\begin{dfn}
[Logarithmic connections] \label{logarithmic connection}
The category of {\it logarithmic connections},  denoted  $\mclog(\pos Rx/R)$, has for     
\begin{enumerate}\item[{\it objects}] those couples $(\cm,\na)$ consisting of a finite   $\pos Rx$-module and an $R$-linear endomorphism $\na:\cm\to\cm$, called  {\it the derivation},  satisfying Leibniz's rule $\na(fm)=\vt(f)m+f\na(m)$, and   
\item[{\it arrows}] from $(\cm,\na)$ to $(\cm',\na')$ are $\pos Rx$-linear morphisms $\ph:\cm\to \cm'$ such that $\na'\ph=\ph\na$.  
\end{enumerate}\end{dfn}

The two categories $\mc(\lau Rx/R)$ 
and $\mclog(\pos Rx/R)$ are abelian categories and there is an evident $R$-linear functor 
\[
\ga:\mclog (\pos Rx/R)\aro\mc(\lau Rx/R).
\]  

\begin{dfn}
[Regular-singular connections and models]
\label{dfn.20220607--05}  \ {}
\begin{enumerate}[(1)]
\item An object $M\in\mc(\lau Rx/R)$ is said to be {\it regular-singular} if it is isomorphic to a certain $\ga(\cm)$ for some $\cm\in \mclog(\pos Rx/R)$.
  The full subcategory of regular-singular connections will be denoted by $\mcrs(\lau Rx/R)$. 
\item Given $M$ in $\mcrs(\lau Rx/R)$, any object 
$\cm\in\mclog(\pos Rx/R)$ such that $\ga(\cm)\simeq M$ is called a {\it logarithmic model} of $M$. 

\item Let $\mcrs^\circ(\lau Rx/R)$, respectively $\mclog^\circ(\pos Rx/R)$, stand for the full subcategory of $\mcrs(\lau Rx/R)$, respectively $\mclog(\pos Rx/R)$, consisting of those objects $(M,\na)$ for which $M$ is a flat $\lau Rx$--module, respectively flat $\pos Rx$--module. 
\end{enumerate} 
\end{dfn}

\begin{rmk}Let $\cm\in\mclog(\pos Rx/R)$ be a model of $M$. Since $(0:x)_\cm\subset\cm$ is preserved by the derivation, it is clear that   $M$ possesses a model $\cm'$ such that $(0:x)_{\cm'}=0$. 
\end{rmk}

\begin{ex}[Euler connections]\label{ex.20220607--01}
Let $(V,A)\in\bb{End}_R$ be given. The logarithmic connection associated with the couple  $(V,A)$ is   defined by the couple $(\pos Rx\ot_RV,D_A)$, where 
$$D_A(f\ot v)= \vt(f)\ot v+f\ot Av.$$
This logarithmic connection is called the  {\it Euler connection associated with $(V,A)$}. Notation: $\mm{eul}_{\pos Rx}(V,A).$ 
\end{ex}

The Euler connections yield a functor, denoted  $\mm{eul}_{\pos Rx}$ or simply $\mm{eul}$ when no confusion may appear:  
\[
\mm{eul}:\bb{End}_R\aro\mclog(\pos Rx/R).
\]
It is straightforward to check that this is an $R$-linear, exact, and faithful tensor functor. Combining $\mm{eul}$ with $\gamma$ we have a functor
$$\ga\mm{eul}:\bb{End}_R\aro
\mcrs(\lau Rx/R).$$

The main aim of this section is to show that this functor produces an equivalence when restricted to objects with {\em exponents} lying in   $\tau\subset C$  (Theorem \ref{thm.2022070-8--01}). We first introduce the exponents.

Let $(\cm,\na)\in\mclog(\pos Rx/R)$. The Leibniz rule implies that $\na(x\cm)\subset x\cm$. Hence, we obtain an $R$-linear endomorphism
\begin{equation}\label{16.05.2020--2}
\mm{res}_\na: \cm/(x)\aro\cm/(x),
\end{equation}
given by 
\begin{equation}\label{16.05.2020--3}
\mm{res}_\na(m+(x) ) = \na(m)+(x).
\end{equation}
Further, taking residue modulo $\mathfrak r$ we have the map 
\begin{equation}
\label{14.07.2022--1}
\ov{\mm{res}}_\na : \cm/(\g r,x)\aro\cm/(\g r,x).
\end{equation}
\begin{dfn}[Residue and exponents]\label{28.06.2021--6}
Let $(\cm,\nabla)\in\mclog(\pos Rx/R)$.
\begin{enumerate}[(1)]
\item 
The $R$-linear map \eqref{16.05.2020--2} is called the {\it residue} of $\na$.  
\item The  eigenvalues of  $\ov{\mm{res}}_\na$ are called the (reduced) {\it  exponents} of $\na$. The set of exponents will be denoted by $\mm{Exp}(\cm,\na)$, $\mm{Exp}(\na)$ or $\mm{Exp}(\cm)$ if no confusion may appear. 
\end{enumerate}
\end{dfn}

The following result was obtained in \cite{HdST22} for $R$ being a complete local $C$-algebra, but the proof works indeed for any local $C$-algebra.

\begin{thm}[{\cite[Theorem 8.10]{HdST22}}]\label{thm.20220607--03} Let $(\cm,\na)\in\mclog(\pos Rx/R)$ be such that $\cm$ is a free $\pos Rx$-module.    
If $\mathrm{Exp}(\cm)\subset\tau$, then  $(\cm,\na)$ is isomorphic to $\mm{eul}_{\pos Rx}(\cm/(x),\mm{res}_\na)$. \qed 
\end{thm}

\subsection{Euler form for connections of $\mcrs^\circ(\lau Rx/R)$}\label{20.09.2023--1}
{\it  We now suppose, until the end of Section \ref{20.09.2023--1}, that  $R$ is in addition  Henselian.} 
With the preparation in the previous subsection, we show now that any regular-singular connection $(M,\nabla)$, where $M$ is a {\it flat}  $R$--module, is isomorphic to an  Euler connection. This is an extension of \cite[Corollary 9.4]{HdST22} to the case where $R$ is solely Henselian. The idea behind the proof is to show that logaritmic models with exponents in $\tau$ exist in all generality (cf. Proposition \ref{thm.20220607--04}) and that these models, when $R$ is complete, are sufficient to characterize the Deligne-Manin models appearing in the central result \cite[Theorem 9.1]{HdST22} (this is the content of Theorem \ref{03.08.2023--1}). Then, basic Commutative 
Algebra (cf. Lemma \ref{faith--flatness}) allows us to find  {\it free} logarithmic models for $(M,\na)$
by using free logarithmic models of $\lau {\wh R}x\ot M$.


In order to present a clear argument, we require the following notations and terminology from \cite{HdST22}. Given $k
\in\NN$ and an object $(\cm,\na)$ of $\mclog(\pos Rx/R)$, or of $\mcrs(\pos Rx/R)$, we let $(\cm,\na)|_k$, or $\cm|_k$ if no confusion is likely, stand for the object of $\mclog(\pos Cx/C)$, respectively $\mcrs(\lau Cx/C)$, obtained from the induced map $\na:\cm/\g r^{k+1}\aro\cm/\g r^{k+1}$. 

We begin by showing that the Deligne-Manin model constructed in \cite[Theorem 9.1]{HdST22} can be singled-out by a much simpler condition.

\begin{thm}\label{03.08.2023--1}We assume that $R$ is complete for the moment. Let 
$(M,\na)\in\mcrs(\lau Rx/R)$ posses a  logarithmic model $\ce\in\mclog(\pos Rx/R)$ enjoying the following properties: 
\begin{enumerate}[(1)]\item All its exponents are on $\tau$.
\item We have  $(0:x)_\ce=0$. 
\end{enumerate}

 Then $\ce$ is isomorphic to the Deligne-Manin logarithmic model described in \cite[Theorem 9.1]{HdST22}. In particular, 

\begin{enumerate}[(a)]
\item the $\pos Cx$--module $\ce|_k$ is free for any given $k\in\NN$.
\item If, in addition $M$ is $R$-flat, then $\ce$ is a \textbf{free} $\pos Rx$-module.  
\end{enumerate}
\end{thm}

\begin{proof}Let $\cm\in\mclog(\pos Rx/R)$ be a model of $M$ as in \cite[Theorem 9.1]{HdST22}. 
Fix $k\in\NN$; we know that $\cm|_k$ is a free   $\pos Cx$-module. Hence, we obtain an arrow $\ph_k:\ce|_k\to\cm|_k$ of $\mclog(\pos Cx/C)$ which fits into 
\begin{equation}\label{03.08.2023--2}\xymatrix{
\ce|_k\ar^-{\ph_k}[rr]\ar[d]_{\text{natural}} && \cm|_k\ar@{^{(}->}[d]^{\text{natural}}
\\M|_k\ar@{=}[rr]&&M|_k
}
\end{equation}
because of \cite[Proposition 4.4(3)]{HdST22}. Note that, $\ph_k$ is the unique arrow rendering diagram \eqref{03.08.2023--2} commutative. 
By this reason and the fact that $\ce$ and $\cm$ are $\g r$-adically complete (see Exercise 8.2 and Theorem 8.7 in  \cite{Mat86}), we obtain an arrow of $\mclog(\pos Rx/R)$ 
\[
\ph:\ce\aro\cm
\]
  enjoying the following properties. 
\begin{enumerate}[(i)]\item For each $k\in\NN$, we have $\ph|_k=\ph_k$. 
\item The following diagram commutes:
\[
\xymatrix{
\ce\ar[rr]^\ph\ar[d]_{\text{natural}}&&\cm\ar[d]^{\text{natural}}
\\
\lip_kM|_k\ar@{=}[rr]&&\lip_kM|_k  .
}
\] 
 \end{enumerate}

Let us now observe that $\ph$ is {\it surjective}. Indeed, $\ce|_0\in\mclog(\pos Cx/C)$ is a model of $M|_0\in\mcrs(\lau Cx/C)$ having exponents in $\tau$ so that, 
\[
\fr{\ce|_0}{(0:x)_{\ce|_0}}
\]
is, being a {\it quotient} of $\ce|_0$, a model of $M|_0$ with exponents in $\tau$. 
 Therefore, the $\pos Cx$--linear mapping 
 \[
 \fr{\ce|_0}{(0:x)_{\ce|_0}}\aro \cm|_0 
 \]
which is induced by $\ph_0$
is an isomorphism \cite[Proposition 4.4(3)]{HdST22}; consequently, $\ph|_0$ is surjective. Because $\pos Rx$,  $\ce$ and $\cm$ are   $\g r$-adically complete, and  because of \cite[$0_{\mathrm I}$, 7.1.14]{ega}, we conclude that $\ph$ is also surjective. 

We need to show that $\ph$ is also {\it injective}. It is tempting to argue with the completion $\lip_kM|_k$, but this is a complicated object and we proceed as at  the end of the proof of \cite[Theorem 9.1]{HdST22}: {\it We   show that \[\ph[x^{-1}]:\ce[x^{-1}]\aro\cm[x^{-1}]\]
is an isomorphism}. Once this is guaranteed, injectivity of $\ph$ is a consequence of the fact that $\ce\to\ce[x^{-1}]$ is injective (by construction) and that 
\[
\xymatrix{\ce\ar@{^{(}->}[d]\ar[rr]^\ph&&\cm\ar[d]
\\
\ce[x^{-1}]\ar[rr]^\sim_{\ph[x^{-1}]}&& \cm[x^{-1}]
}
\]
commutes. 

That $\ph[x^{-1}]$ is an isomorphism is verified by
the ensuing arguments. We start by
 observing that $\ph[x^{-1}]|_k$ is an isomorphism for all $k$, so that $\ph[x^{-1}]$ is an isomorphism in a {\it neighborhood of the closed fiber of $\spc\,\lau Rx\to\spc \,R$}. This implies that $N:=\mm{Ker}\,\ph[x^{-1}]$ and $Q:=\mm{Coker}\,\ph[x^{-1}]$ vanish on an open neighborhood of the aforementioned closed fiber. Now, $N$ and $Q$ are objects of $\mc(\lau Rx/R)$, so that for each $\g p\in\spc\,R$, the fibers $N\ot_R\bm k(\g p)$
and $Q\ot_R\bm k(\g p)$ are flat as $\lau Rx\ot_R\bm k(\g p)$-modules \cite[Theorem 8.19]{HdST22}. A simple argument in Commutative Algebra \cite[Lemma 9.2]{HdST22} now shows that $N=Q=0$, thus assuring that $\ph[x^{-1}]$ is an isomorphism. 

To end, if $M$ is $R$-flat, Corollary {9.4} of \cite{HdST22} is enough to conclude the prof of item (b). 
\end{proof}

In view of Theorem \ref{03.08.2023--1}, it becomes important, even when $R$ is solely Henselian, to construct logarithmic models with exponents in $\tau$. The construction follows the classical method of using ``Jordan subspaces'' (generalized eigenspaces) to adjust the exponents \cite[Section 17.4]{Was76} but, in the present case it is necessary to have such a decomposition for $R$-linear endomorphisms. This is a consequence of the following lemma which was mentioned in Remarks 8.15(a) of \cite{HdST22}.

\begin{lem}[``Jordan decomposition over $R$'']\label{lem.202307--01}
		Let $(V,\ph)\in \bb{End}_R$ 
and denote by $\ov\ph:\ov V\to\ov V$ the reduction of $\ph:V\to V$ modulo $\g r$. Let 
 $\{\vr_1,\ldots,\vr_r\}$ be the spectrum of $\ov{\ph}$ and write  
\[
\ov V=\bigoplus\limits_{i=1}^r \mm{Ker}\big(\ov \ph-\vr_i\big)^{\mu_i}
\]
Then,  there exists a direct sum
$$V=V_1\oplus\ldots\oplus V_r,$$ 
where $V_i$ is $\ph$--invariant $R$--submodule of $V,$ such that its reduction modulo $\g r$ is $\mathrm{Ker}\, \big(\ov \ph-\vr_i\big)^{\mu_i},$ for each $1 \leq i \leq r$.
\end{lem}

\begin{proof}
Let $R^n\to V$ be a surjection inducing an isomorphism   $C^n\to V/\mathfrak{r}$. Then $\varphi$ 
lifts to $\tilde\varphi:R^n\to R^n$ and the residue of the characteristic
polynomial of $\tilde{\varphi}$  equals the characteristic polynomial  of $\ov{\varphi}$:
$$\ov{P_{\tilde{\varphi}}(T)}=P_{\ov\varphi}(T).$$
As $R$ is  Hensenlian, the factorization
$$P_{\ov\varphi}(T)= \prod\limits_{i=1}^r\big(T-\vr_i\big)^{\mu_i}$$
lifts to a factorization
\[
P_{\tilde{\varphi}}(T)= \prod\limits_{i=1}^{r}g_i(T),
\]
where, for any $1\le i\le r$, the polynomials $g_i$ and $\hat g_i:=\prod_{j\not=i}g_j$ are {\it strictly coprime}, i.e. 
\[
\tag{$\star$} R[T]\po g_i+R[T]\po\hat g_i=(1),
\]
cf. \cite[Ch. I, Section 4, p.32]{milne}.  
Let $V_i=\mathrm{Ker}\, g_i(\ph)$; then $V_i$ are $\ph$--invariant 
$R$--submodules of $V$. From $(\star)$
and the fact that $P_{\tilde \ph}(\ph)$ vanishes identically on $V$, it is easy to see that 
\[
\begin{split}V&=\mm{Ker}\,P_{\tilde\ph}(\ph)
\\&=V_1\op\cdots\op V_r.
\end{split}
\]
Now, the composition $V_i\to V\to \ov V$ sends $V_i$ to $\mm{Ker}(\ov\ph-\vr_i)^{\mu_i}$ and has kernel $\g rV\cap V_i=\g rV_i$. It is easily verified that $V_i\to \mm{Ker}(\ov\ph-\vr_i)^{\mu_i}$ must  be surjective as well, so that the last claim is   verified.  
\end{proof}

\begin{ex}If we drop the assumption that $R$ be Henselian, the above result certainly fails. Suppose for example that $R=\{a/b\,:\,a,b\in C[t],\,b(0)\not=0\}$. Define $\ph=\begin{pmatrix}0&1+t\\1&0\end{pmatrix}:R^2\to R^2$. Then $\ov\ph:C^2\to C^2$ acts by multiplication by $1$ on $C(\vec e_1+\vec e_2)$ and by multiplication by $-1$ on $C(\vec e_1-\vec e_1)$. On the other hand, it is not possible to decompose $R^2$ into a direct sum of submodules of rank one.  
\end{ex}
	
We are now ready to show the existence of   logarithmic models having exponents on $\tau$.
\begin{prp}[Shearing] \label{thm.20220607--04} Let $(M,\na)$ be the regular--singular connection on $\lau Rx/R$. Then there exists a logarithmic model 
$\cm\in \mclog(\pos Rx/R)$ of $(M,\na)$ such that $\mm{Exp}(\cm)\subset\tau$ and $(0:x)_\cm=0$.
	\end{prp} 
\begin{proof} 
Let $\ce$ be an arbitrary logarithmic model of $(M,\na)$ such that $(0:x)_\ce=0$ ---  and hence $\ce\subset M$. We shall proceed by reverse induction on the non-negative integer 
\[b(\ce):=\#\,\,\mm{Exp}(\ce)\setminus\tau.\]
Let us assume that $b(\ce)>0$ and let $\vr\in\mm{Exp}(\ce)\setminus\tau$. We shall construct a logarithmic model $\ce'$ such that $b(\ce')<b(\ce)$ and $(0:x)_{\ce'}=0$. Let $V:=\ce/x\ce$, this is an $R$-module, and consider its decomposition into ``generalized eigenspaces'', with respect to the residue morphism $\mm{res}_\ce$, 
\[
V=\bigoplus_{\si\in\mm{Exp}(\ce)} V_\si
\]
as in Lemma \ref{lem.202307--01}. 
(Each $V_\si$ is an $R$-module.) Note that there exists $\mu\in\NN$ such that, for each $\si$ and $v\in V_\si$, we have 
\[
(\mm{res}_\ce-\si)^\mu v\in \g rV_\si.
\]
In particular
\[
\prod_{\si\in\mm{Exp}(\ce)}(\mm{res}_\ce-\si)^\mu(V)\subset \g rV.
\]
The reduction map $\ce\to V$ shall be denoted by $e\mapsto\tilde e$. Then, if $\ce_\si:=\{e\in\ce\,:\,\tilde e\in V_\si\}$, we have 
\[\ce=\sum_{\si}\ce_\si.\]
Each $\ce_\si$ is stable under $\na$ because $\wt{\na e}=\mm{res}(\tilde e)$. 
 In addition, each $\ce_\si$ is   an $\pos Rx$-submodule of $\ce$ and 
 \[\tag{$\star$}
 (\na-\si)^\mu(\ce_\si)\subset(x,\g r)\ce_\si.
 \] 
Let $k\in\ZZ$ be such that $\vr+k\in\tau$. Define $\ce_\vr'=x^k\ce_\vr$ and 
\[\ce':=\ce_\vr'+\sum_{\si\not=\vr}\ce_\si,\]
which is an $\pos Rx$-submodule of $M$, stable under the action of $\na$. 
We now choose $e\in\ce_\vr$ and let $e':=x^ke\in\ce_\vr'$. By a direct verification, we know that 
\[
[\na-(\vr+k)]^\mu(e')=x^k[\na-\vr]^\mu(e).
\]
Because $[\na-\vr]^\mu(e)\in(x,\g r)\ce_\vr$, we then have 
\[\tag{$\dagger$}
[\na-(\vr+k)]^\mu(e')\in x^k\po(x,\g r)\ce_\vr\subset(x,\g r)\ce'.
\]
Consequently, from ($\star$) and ($\dagger$), the $R$--linear map 
\[
\left[\na-(\vr+k)\right]^\mu\prod_{\si\not=\vr}(\na-\si)^\mu
\]
sends $\ce'$ into $(x,\g r)\ce'$. Letting $V'=\ce'/x\ce'$, we conclude that  
\[ \left[\mm{res}_{\ce'}-(\vr+k)\right]^\mu\prod_{\si\not=\vr}(\mm{res}_{\ce'}-\si)^\mu(V')\subset\g rV',
\]
showing that $\mm{Exp}(\ce')\subset(\mm{Exp}(\ce)\setminus\{\vr\})\cup\{\vr+k\}$, which in particular proves that $b(\ce')<b(\ce)$. Obviously, $\ce'$, being contained in $M$ is such that $(0:x)_{\ce'}=0$.

\end{proof}

We now require a result in Commutative Algebra. 

\begin{lem}\label{faith--flatness}The following claims are true.
\begin{enumerate}[(i)]
\item The homomorphisms 
$\pos Rx\to \pos {\wh R}x$ and
$\lau Rx\to \lau {\wh R}x$ are faithfully flat.
\item A finite $\lau Rx$-module $E$  is flat  if and only if $\lau{\wh R}x\ot_{\lau Rx}E$ is $\lau{\wh R}x$-flat. A finite $\pos Rx$-module $\ce$ is free if and only if $\pos{\wh R}x\ot_{\pos Rx}\ce$ is $\pos{\wh R}x$--flat. 
\end{enumerate}
\end{lem}
\begin{proof}
(i) Firstly,  $\pos {\wh R}x$ is $(\g r,x)$--adically complete \cite[Exercise~8.6]{Mat86}. Thus, we can view $\pos {\wh R}x$ 
as the $(\g r,x)$--adic completion of $\pos Rx$. 
As $\pos Rx$ is a noetherian local ring, we conclude that  $\pos {\wh R}x$ is
faithfully flat over $\pos Rx$   \cite[Theorem~8.14]{Mat86}. The fact that $\lau Rx\to \lau {\wh R}x$ is faithfully flat is a consequence of the fact that this mapping is obtained from $\pos Rx\to \pos {\wh R}x$ by inverting $x$.

(ii) This is \cite[Exercise 7.1]{Mat86} together with the fact that a finite module over a local noetherian ring is flat if and only if it is free \cite[Theorem 7.10]{Mat86}.
\end{proof}

\begin{thm}\label{thm.230704}
Let  $(M,\na)$ be a regular--singular connection  of $\lau Rx/R$, with $M$ being
flat as a $\lau Rx$--module. (That is, $(M,\na)$ is an object of $\mcrs^\circ(\lau Rx/R)$.) Then, $M$ possesses a logarithmic   model $\cm$ which, as an $\pos Rx$-module, is \textbf{free}, and in particular $M$ is a free $\lau Rx$-module. Moreover, the model $\cm$ can be chosen of the form $\mm{eul}(V,A)$, with $\mm{Exp}(\cm)\subset\tau$.
\end{thm} 

\begin{proof}
Let $\cm\in\mclog(\pos Rx/R)$ be a model of $M$ 
as in Proposition \ref{thm.20220607--04}.  Then, $\pos{\wh R}x\ot_{\pos Rx}\cm\in\mclog(\pos{\wh R}x/R)$ is a model of $\lau{\wh R}x\ot_{\lau Rx}M$ as in Theorem \ref{03.08.2023--1} and hence, $\lau{\wh R}x\ot_{\lau Rx}M$ being flat over $\wh R$, it must be that $\pos{\wh R}x\ot_{\pos Rx}\cm$ is free over $\pos{\wh{R}}x$. It then follows that $\cm$ is $\pos Rx$-free, by Lemma \ref{faith--flatness}-(ii). Consequently, $M$ is free over $\lau Rx$. To verify the last claim, it suffices to employ Theorem \ref{thm.20220607--03}.
\end{proof}

\begin{thm}\label{thm.2022070-8--01}
The functor  
\[
\ga \mm{eul}_{\pos Rx}:\bb{End}_R^{{\circ}}\aro\mcrs^{{\circ}}(\lau Rx/R)
\] is faithful and essentially surjective. This functor is not full. Assume that $0\in\tau$; 
then
  restriction of $\ga \mm{eul}_{\pos Rx}$ to the full subcategory of all objects $(V,A)$ such that the spectrum of $A:V/\g r\to V/\g r$ is contained in  $\tau$,  is indeed full. 
\end{thm}
\begin{proof}
 Essential surjectivity   is already verified by Theorem  \ref{thm.230704}, while  faithfulness is obvious. We then   concentrate on the verification of the last claim. Let $(V,A)$ and $(W,B)$ be objects of $\bb{End}_R^\circ$ and suppose that the eigenvalues of the $C$-linear endomorphisms   $A_0:V/\g r\to V/\g r$ and $B_0:W/\g r\to W/\g r$ associated respectively to $A$ and $B$ lie in $\tau$. On  $H=\mm{Hom}_R(V,W )$, consider the endomorphism $T:h\mapsto hA-Bh$; we then obtain an   object $(H,T)$ of $\bb{End}^\circ_R$. The spectrum of the $C$-linear endomorphism $T_0 :H/\g r\to H/\g r$ is built up from the differences of eigenvalues of $A_0$ and $B_0$  \cite[II, Problem 4.1]{Was76},  so that   $\mm{Sp}_{T_0}\cap\ZZ\subset\{0\}$. Consequently,
 for each $k\in\NN$, the spectrum of the $C$-linear endomorphism $T_k:H/\g r^{k+1}\to H/\g r^{k+1}$ contains no integers except perhaps 0.  This is because $\mm{Sp}_{T_k}=\mm{Sp}_{T_0}$ \cite[Prp. 8.11]{HdST22}. 
It is a simple matter to see that  $\hh{\mc}{\ga \mm{eul}(V,A)}{\ga \mm{eul}(W,B)}$ corresponds to the horizontal elements of $\ga \mm{eul}(H,T)$. After picking a basis of $H$, a horizontal section of $\ga eul(H,T)$ amounts to a vector $\bm h \in  \lau Rx^{r}$ 
such that 
\[
\vt\bm  h=-T\bm h.
\]
Writing $\bm h=\sum_{i\ge i_0}\bm h_ix^i$, we see that 
\[
T\bm h_i=-i\bm h_i.
\]
Now, let $i\not=0$. Then the image of  $\bm h_i$ in $R_{k}^{\op r}$ must be zero, since $i\not\in\mm{Sp}_{T_k}$. Hence, $\bm h_i=0$ \cite[Theorem 8.10(i)]{Mat86} except perhaps for $i=0$. This proves that any arrow 
\[
h:\ga \mm{eul}(V,A)\aro\ga \mm{eul}(W,B)
\]
comes from an arrow $V\to W$. 
\end{proof}

\subsection{The case where $R$ is a DVR}
Previously,  we   described the objects of $\mcrs^\circ$, but we still have no conclusions in general. So let us, in this section, add to the assumption that $R$ is Henselian the hypothesis 
\[\text{$R$ is a DVR and $\g r=Rt$.}\]
In this setting, we shall show that the functor
$\ga\mm{eul}_{\pos Rx}:\bb{End}_R\to \mcrs(\lau Rx/R)$ is \textit{essentially surjective}. See Corollary \ref{cor.12.07.22-1}. Part of this result was already achieved by Theorem \ref{thm.2022070-8--01}. It is a technique from \cite{DH18}, expressed in Proposition \ref{prop.20230730} below, which   allows the deduction of the general case.

\begin{prp}\label{prop.20230730}
Each object of $\mcrs(\lau Rx/R)$ is a quotient of a certain $(E,\na)\in\mcrs(\lau Rx/R)$ such that $E$ is a free $\lau Rx$-module.   
\end{prp}
\begin{proof}
The proof is almost identical to that  of \cite[Proposition~5.2.2]{DH18}, but some  care has to be taken to assure that the connections constructed are regular-singular. 

Let $(M,\na)\in\mcrs(\lau Rx/R)$ be given. The reader should recall that in view of Proposition \ref{rmk.20220607--01}, a necessary and sufficient condition for $N\in\mcrs(\lau Rx/R)$ to belong to $\mcrs^\circ(\lau Rx/R)$ is that $(0:t)_N=0$. 

Let us introduce, for a given finite $\lau Rx$-module $W$, the submodule 
\[
\begin{split}W_{\mathrm{tors}}&=\bigcup_{k=1}^\infty(0:t^k)_W
\\
&=\{w\in W\,:\,\text{some power of $t$ annihilates $w$}\}.
\end{split}
\]
Noetherianity assures that $W_{\mm{tors}}=(0:t^\ell)_W$ for some $\ell$ and we define 
\[
\begin{split}
r(W)&=\min\{k\in\NN\,:\,t^kW_{\mm{tors}}=0\}\\&=\min\{k\in\NN\,:\,W_{\mm{tors}}=(0:t^k)_W\}.
\end{split}
\]
The proposition is to be proved by induction on $r(M)$. (Note that for each $k$, the submodule $(0:t^k)_M$ is stable under $\na$.)
If $r(M)=0$, then $M_{\mm{tors}}=0$ and there is nothing to be done. Assume $r(M)=1$, so that $(tM)_{\mm{tors}}=0$. Let $q:M\to Q$ be the quotient by $tM$; since $Q$ is annihilated by $t$, this is an object of $\mcrs(\lau Cx/C)$ and as  such has the form $\ga\mm{eul}(V,A)$, where $V$ is a $C$-vector space \cite[Cor. 4.3]{HdST22}. This connection is certainly a quotient of the Euler connection 
\[\wt Q:=\ga\mm{eul}(R\ot_CV,\id_R\ot A),\]
which is an object of $\mcrs(\lau Rx/R)$. We then have a diagram with exact rows: 
\[
\xymatrix{0\ar[r]& t M\ar[r]&M\ar[r]^q&Q\ar[r]&0\\ 0\ar[r]& t M\ar[u]^\sim\ar[r]&\wt M\ar@{}[ru]|\square\ar[u] \ar[r]&\ar@{->>}[u] \wt Q\ar[r]&0,}
\]
where the rightmost square is cartesian and $\wt M\to M$ is in fact  surjective. Since $(t M)_{\mm{tors}}=\wt Q_{\mm{tors}}=0$, we have  $\wt M_{\mm{tors}}=0$. Since $\wt M$
is a subobject of $\wt Q\op M$,  we can appeal to \cite[Proposition~8.3]{HdST22} to assure that it is   regular-singular. In conclusion, $\wt M\in\mcrs^\circ(\lau Rx/R)$. 

Let us now assume that $r(M)>1$. Let $N=(0:t)_M$ and observe that $r(N)=1$. Denote by $q:M\to Q$ the quotient by $N$. It then follows that $t^{r(M)-1}Q_{\mm{tors}}=0$, so that $r(Q)\le r(M)-1$. By induction, there exists $\wt Q\in\mcrs^\circ(\lau Rx/R)$ and a surjection $\wt Q\to Q$. We arrive at a commutative diagram with exact rows
\[
\xymatrix{0\ar[r]& N\ar[r]&M\ar[r]^q&Q\ar[r]&0\\ 0\ar[r]& N\ar[u]^\sim\ar[r]&\wt M\ar@{}[ru]|\square\ar[u] \ar[r]&\ar@{->>}[u]\wt Q\ar[r]&0,}
\]
where the rightmost square is cartesian so that $\wt M\to M$ is surjective.  Since $\wt Q_{\mm{tors}}=0$, we conclude that $\wt M_\mm{tors}= N$, so that $r(\wt M)=1$. We can therefore find  $\wt M_1\in\mcrs^\circ(\lau Rx/R)$ and a surjection  $\wt M_1\to\wt M$ and consequently  a surjection $\wt M_1\to M$.
\end{proof}

\begin{cor}\label{cor.12.07.22-1}
The functor $\ga\mm{eul}_{\pos Rx}:\bb{End}_R\to \mcrs(\lau Rx/R)$ is \textit{essentially surjective}. 
\end{cor}
\begin{proof} 
We assume that $0\in\tau$. Let $M\in\mcrs(\lau Rx/R)$ be given; because of Proposition \ref{prop.20230730}, we can find an exact sequence in $\mcrs(\lau Rx/R)$: 
\[
E\stackrel\Ph\aro F\aro M\aro 0
\]
where $E$ and $F$ belong to $\mcrs^\circ(\lau Rx/R)$. According to Theorem \ref{thm.2022070-8--01}, we can assume that 
\[
E= \ga{\rm eul}_{\pos Rx}(V,A)\quad
\text{and}
\quad F=\ga{\rm eul}_{\pos Rx}(W,B),
\]  
where $(V,A)$ and $(W,B)$ belong to $\bb{End}_R^{\circ}$ and the spectra of 
\[
V/(t)\arou A V/(t)\quad
\text{and}
\quad
W/(t)\arou B W/(t)
\]
 are all contained in $\tau$. In this case, $\Ph=\ga\mm{eul}_{\pos Rx}(\ph:V\to W)$, again by Theorem \ref{thm.2022070-8--01}, and hence $M$ is isomorphic to $\ga\mm{eul}_{\pos Rx}(\mathrm{Coker}\,\ph)$.
\end{proof}

\section{Structure of $\mcrs(R[x^\pm]/R)$}\label{sect.4}
Our aim in this section is to relate $\bb{End}_R$ and $\mcrs(R[x^\pm]/R)$
and obtain the equivalent of Corollary \ref{cor.12.07.22-1} in this setting. Our strategy is different from the one in the previous section. Instead of using the shearing technique, we rely on Popescu's approximation theorem to descend from $\wh R$ to $R$.

We fix a choice of local coordinates of $\PP_R^1$ as follows: write $\PP_R^1$ as the union of two affine lines $\mathds A_0$ and $\mathds A_{\infty},$ where $\mathds A_0=\mathrm{Spec}(R[x])$ and $\mathds{A}_{\infty}=\mathrm{Spec}(R[y]),$ with the transition function on their intersection $R[x^\pm]= R[y^\pm]$ being $y=x^{-1}$.  

By the equality $y=x^{-1}$ we have
\[
x\frac{d}{dx}=-y\frac{d}{dy},
\]
and therefore $\vartheta:R[x^\pm]\to R[x^\pm]$ can be extended canonically to a global section, still denoted by $\vartheta,$ of the tangent sheaf of $\PP_R^1$. 

\begin{dfn}[Connection on the punctured affine line]\label{dfn.20220607--08}
The category of  $R$--connections on $R[x^\pm]$,  or of connections on $R[x^\pm]/R$, or on 
the punctured affine line
 $\PP_R^1\setminus\{0,\infty\}$, etc,  denoted $\mc(R[x^\pm]/R)$, has for   
 
 \begin{enumerate}\item[\it objects] those couples $(M,\na)$ consisting of a \textit{$R[x^\pm]$--module of finite presentation} and an $R$--linear endomorphism $\na:M\to M$ satisfying Leibniz's rule \[\na(fm)=\vt(f)m+f\na(m);\] 
\item[\it arrows] between $(M,\na)$ and $(M',\na')$ are just $R[x^\pm]$--linear maps $\ph:M\to M'$ satisfying $\na'\ph=\ph\na$.  
\end{enumerate}  
\end{dfn}

It is well-known that for a connection $(M,\nabla)$ on $R[x^\pm]/R$, 
a necessary and sufficient condition for $M$ to be $R[x^\pm]$--flat is that it be $R$-flat, cf., e.g. \cite[p.82]{dS09} or \cite[Proposition 5.1.1]{DH18}. (We profit to note that in the proof of Proposition 5.1.1 in \cite{DH18}, we need to employ the ``fiber-by-fiber flatness criterion'' \cite[$\mathrm{IV}_3$, 11.3.10]{ega} and not the ``local flatness criterion.'') Moreover, although these references are written in the context where $R$ is a DVR, the idea of proof applies in more generality since it is a consequence of the ``fiber-by-fiber flatness criterion''  \cite[$\mathrm{IV}_3$, 11.3.10]{ega} and the well-documented case of a base field of characteristic zero. See \cite[Remark 8.20]{HdST22} for more details and references.

\begin{dfn}[Logarithmic connections on the punctured affine line] 
\label{dfn.20220607--08b}
The category of {\it logarithmic connections} on the punctured affine line, denoted $\mclog(\PP_R^1/R)$, has for
\begin{enumerate}\item[\it objects] those couples $(\cm,\na)$ consisting of a coherent $\co_{\PP_R^1}$--module and an $R$--linear endomorphism $\na:\cm\to \cm$ satisfying Leibniz's rule $\na(fm)=\vt(f)m+f\na(m)$ on all open subsets; and
\item[\it arrows] between $(\cm,\na)$ and $(\cm',\na')$ are   $\co_{\PP_R^1}$--linear maps $\ph:\cm\to \cm'$ satisfying $\na'\ph=\ph\na$.  
\end{enumerate}  
\end{dfn}

We let
\[
\ga: \mclog(\PP_R^1/R)\aro\mc(R[x^\pm]/R)
\]
be the natural restriction functor.

\begin{dfn}
[Regular-singular connections on the punctured affine line]\ {}
\begin{enumerate}[(1)]
\item 
A connection $(M,\na)$ in $\mc(R[x^\pm]/R)$ is  {\it regular-singular} if $\ga(\cm)\simeq M$ for a certain $\cm\in\mclog(\PP_R^1/R)$; in this case, any such $\cm$ is a  {\it logarithmic model} of $M$. 
\item The full subcategory of $\mc(R[x^\pm]/R)$ having regular-singular connections as objects   is denoted by $\mcrs(R[x^\pm]/R)$.
\item The full subcategory of $\mc(R[x^\pm]/R)$ having as objects those connections $(M,\na)$
with $M$ being a flat $R[x^\pm]$-module is denoted by $\mcrs^\circ(R[x^\pm]/R)$. 
\end{enumerate}
\end{dfn}

 The prime example of regular-singular connections is described now: 

\begin{ex}[Euler connections]\label{ex.20220709--01}
For an object $(V,A)\in\bb{End}_R$,  we set 
$$\mm{eul}_{\PP^1_R}(V,A):= (\mathcal{O}_{\PP^1_R} \otimes_R V, D_A),$$
where $D_A: \mathcal{O}_{\PP_R^1} \otimes_R V \rightarrow \mathcal{O}_{\PP_R^1} \otimes_R V$ is $R-$linear and defined by 
$$D_A(f\otimes m) = \vartheta (f) \otimes v + f \otimes Av$$
 on any open subsets of $\PP_R^1$. 
\end{ex}

Thus we have a functor  
$
\mm{eul}_{\PP_R^1}:\bb{End}_R\aro\mclog(\PP_R^1/R)$
and, composing it with $\gamma$, another the functor
\[\gamma\mm{eul}_{\PP_R^1}: \bb{End}_R\aro
\mcrs(R[x^\pm]/R).
\]

For the next theorem, we shall require the notion of $G$-rings \cite[Section 32]{Mat86}.  A field is a $G$-ring  as is a   discrete valuation ring of characteristic zero. Other   relevant $G$-rings are noetherian complete local rings \cite[Theorem 32.3]{Mat86},  rings of finite type over  $G$-rings \cite[33.G]{mat2}  and Henselizations of local $G$-rings (use   \cite[Theorem 32.1]{Mat86} and \cite[Theorem 32.2]{Mat86}). That this concept is necessary here comes from its role in the Popescu approximation theorem. 
 
\begin{thm}\label{prop.20220607--03}
We assume that $R$ is   Henselian in all that follows. 
\begin{enumerate}[(i)] \item Suppose that $R$ is a  $G$-ring.  Then the functor 
\[
\ga \mm{eul}_{\PP_R^1}: 
\bb{End}_R^\circ  \aro \mcrs^\circ(R[x^{\pm}]/R)
\]
 is faithful and essentially surjective.
 \item Suppose that  $R$ is a discrete valuation ring. (In which case $R$ is also a G-ring.) Then the functor     
\[
\ga \mm{eul}_{\PP_R^1}: 
\bb{End}_R   \aro \mcrs(R[x^{\pm}]/R)
\]
is faithful and essentially surjective. 
\end{enumerate}
\end{thm}
\begin{proof}Faithfulness is obvious, in any case, and we proceed to verify essential surjectivity. We shall deal with cases (i) and (ii) at the same time.  
The idea is to  first base change to $\wh R$, use the known results from \cite{HdST22}, and then descend back to $R$ by means of Popescu's theorem.

The map $R\to\wh R$ is regular, by assumption in case (i), and because $R$ is of characteristic zero in case (ii). According to Popescu (see \cite[Theorem 2.5]{Pop86} or   \cite[Theorem 1.1]{spivakovsky99}),
\[
\wh R=\lid_{\la\in L}S_\la
\]
where each $S_\la$ is a {\it smooth} $R$--algebra.

Let $(M,\na)\in\mc_{\rm rs}(R[x^\pm]/R)$.
For each $\la\in L$, we let $(M_\la,\na_\la)$ stand for the object of $\mc(S_\la[x^\pm]/S_\la)$ defined, in an evident manner, by employing the functor $S_\la\ot_R(-)$. We define $(\wh M,\wh\na)$ in similar fashion.

Let us for a moment assume that $(M,\na)\in\mcrs^\circ(R[x^\pm]/R)$ to treat case (i). 
Let $\g A:\g V\to \g V$ be an endomorphism of a certain finite $\wh  R$-module $\g V$ such that  there exists an {\it isomorphism}
\begin{eqnarray*}
  (\wh M,\wh\na)&\arou{\g f}&(\wh R[x^\pm]\ot_{\wh R}\g V,D_{\g A})\\
  &&=\ga\mm{eul}_{\PP_R^1}(\g V,\g A)
\end{eqnarray*}
 in $\mc(\wh R[x^\pm]/\wh R)$. The existence of this arrow is a consequence of \cite[Theorem~10.1]{HdST22} and Theorem \ref{thm.2022070-8--01}. 
Clearly, $(\g V,\g A)\in\bb{End}_R^\circ$ in this case.  
 
 If we now drop the assumption that $(M,\na)\in\mcrs^\circ(R[x^\pm]/R)$, but decree that $R$ is a DVR in order to work on case (ii), then  \cite[Theorem~10.1]{HdST22},
 in conjunction with 
 Corollary \ref{cor.12.07.22-1}, ensure the existence of the isomorphism  $\g f$ as before. 
Granted the existence of $\g f$, the hypothesis in (i) and (ii) now have little bearing on what follows. 

There exists $\al$ such that $\g A:\g  V\to\g V$ is of the form
\[
\id_{\wh R}\otu{S_\al}A_\al: \wh R\otu{S_\al}V_\al\aro \wh R\otu{S_\al}V_\al
\]
where $A_\al$ is an $S_\al$-linear endomorphism of the finite $S_\al$-module $V_\al$, see \cite[$\mathrm{IV}_3$, 8.5.2(i)-(ii), p.20]{ega}.
Given $\la\ge\al$, let $A_\la:V_\la\to V_\la$ be the base-changed endomorphism 
\[
\id\ot A_\al:S_\la\otu{S_\al}V_\al\aro S_\la\otu{S_\al}V_\al.
\]

This allows us to define objects   
\[
(S_\la[x^\pm]\ot_{S_\la}V_\la,D_{A_\la}) 
\] 
from $\mc(S_\la[x^\pm]/S_\la)$,  for all $\la\ge\al$, along the lines of Example \ref{ex.20220709--01}.

There exists $\be\ge\al$ such that $\g f$ is obtained from  a certain mapping of $S_\be[x^\pm]$--modules
\[
 f_\be\,:\,   M_\be   \,\,\aro\,\,   S_\be[x^\pm]\otu{S_\al}V_\al
\]
by base change $S_\be\to \wh R$, see \cite[$\mm{IV}_3$, 8.5.2.1, p.20]{ega}. Note, in addition, that $f_\be$ can be taken to be an {\it isomorphism} of $S_\be[x^\pm]$-modules. 
Let $f_\la$ be the base change of $f_\be$ for $\la\ge\be$.

 Let now $\{m_i\}\in  M$ be a  set of $R[x^\pm]$-module generators for $M$ and write
 $m^\la_i$ for the image of $m_i$ in $M_\la$ via the natural arrow $M\to M_\la$.
Consider, for each $\la\ge\be$, the elements 
\[
\de_i^\la:=f_\la( \na_\la (m^\la_{i}) ) -  D_{A_\la}(f_\la(m^\la_i) )
\]
of $S_\la[x^\pm]\ot_{S_\al}V_\al$. We then conclude that for some $\ga\ge\be$, the elements $\de_i^\ga$ are all zero, and hence  the arrow
\[f_\ga: M_\ga\aro S_\ga[x^\pm]\ot_{S_\al} V_\al\]
is horizontal, as  is verified without much effort.

Because $C$ is algebraically closed, it is clear that $C\to C\ot_RS_\ga$ has a section $C\ot_RS_\ga\to C$ and hence ``Hensel's Lemma'' \cite[$\mm{IV}_4$, Theorem 18.5.17]{ega} shows that $R\to S_\ga$ also has a section $\xi_\ga:S_\ga\to R$. Base changing the morphism $f_\ga$ through $\xi_\ga$, we obtain an isomorphism of connections  $M\to R[x^\pm]\ot V$.
It is clear that if $M$ is $R$-flat, then $V$ is also $R$-flat. 
\end{proof}

\begin{cor}Let us instate the assumptions of Theorem \ref{prop.20220607--03}. Then, if $(M,\na)$ is an object of $\mcrs^\circ(R[x^\pm]/R)$, it follows that $M$ is in fact a free $R[x^\pm]$--module. \qed
\end{cor}

\begin{cor}\label{18.09.2023--1}Let us instate  the assumptions of Theorem \ref{prop.20220607--03}-(ii). Then, each object of $\mcrs(R[x^\pm]/R)$ is a quotient of an object of    $\mcrs^\circ(R[x^\pm]/R)$.  
\end{cor}
\begin{proof}Any object of $\bb{End}_R$ is a quotient of an object of $\bb{End}_R^\circ$ and the result follows from Theorem \ref{prop.20220607--03}. 
\end{proof}

\section{Deligne's equivalence}
\label{sect.5}

In this section, we put things together to obtain an analogue of Deligne's equivalence in the case of a  strict Henselian discrete valuation ring. Recall that Deligne proved in \cite[Proposition~15.35]{Del87} that for any field $k$ of characteristic $0$, the functor 
\[
\bb r: \mcrs(k[x^\pm]/k)
\longrightarrow\mcrs(\lau kx/k) 
\]
given by base change is indeed an equivalence. 
When  $k$ is  replaced by a   $C$-algebra of the form $\pos C{t_1,\ldots,t_r}/\g a$, the analogous equivalence has been established  in \cite{HdST22}.  We want to establish an analogue in the case where $k$ is replaced by our $R$ (assumptions on it will be made as needed). 
\begin{thm}\label{thm.20220623--01}
Let $R$ be Henselian $G$-ring. Then the restriction functor 
\[
\bb r: \mcrs^{\circ}(R[x^\pm]/R)\aro  \mcrs^{\circ} (\lau Rx/R)	
\] 
is an equivalence.
\end{thm}
\begin{proof}
{\em Essential surjectivity.}
This follows from Theorem \ref{thm.2022070-8--01} without much difficulty since $\bb r\ga\mm{eul}_{\PP_R^1}(V,A)=\ga\mm{eul}_{\pos Rx}(V,A)$.

{\it Faithfulness.} This is simple, as for any $N\in\mcrs^\circ(R[x^\pm]/R)$, the natural map 
\[
R[x^\pm]\ot_{R[x^\pm]}N \aro \lau Rx\ot_{R[x^\pm]}N\] is injective.

{\it Fullness.} By Theorem \ref{prop.20220607--03}-(i), we need to prove the following. Let $(V,A)$ and $(W,B)$ be objects of $\bb{End}_R^\circ$. Then the natural map 
\[
\mathrm{Hom}(\ga\mm{eul}_{\PP_R^1}(V,A)\,,\,\ga\mm{eul}_{\PP_R^1}(W,B))\aro \mathrm{Hom}(\ga\mm{eul}_{\pos Rx}(V,A)\,,\,\ga\mm{eul}_{\pos Rx}(W,B))
\]
is surjective. Fixing bases $\{v_i\}_{i=1}^m$, resp. $\{w_i\}_{i=1}^n$,
of $V$, resp. $W$, over $R$, any 
\[
\ph\in\mathrm{Hom}(\ga\mm{eul}_{\pos Rx}(V,A)\,,\,\ga\mm{eul}_{\pos Rx}(W,B)
\]
is defined by an $n\ti m$ matrix $\Ph$ with coefficients in $\lau Rx$. On the other hand, after  base-changing  to $\lau{\wh R}x$, \cite[Theorem 10.1]{HdST22} 
tells us that $\Ph\in\mm{Mat}_{n\ti m}(\wh R[x^\pm])$. Since $\lau Rx\cap \wh R[x^\pm]= R[x^\pm]$, we are done.
\end{proof}

As as consequence, we obtain a ``full'' Deligne equivalence as follows.
\begin{cor}\label{mainthm}
	\label{cor.deligne.equivalence.hensel}
If $R$ is a Henselian discrete valuation ring,	then the restriction functor
	\[	\bb r: \mcrs(R[x^\pm]/R)\aro  
	\mcrs(\lau Rx/R)	\] 
	is an equivalence.
\end{cor}
\begin{proof}
	 It is clear that Corollary \ref{cor.12.07.22-1} implies that the functor $\bb r$ is {\it essentially surjective}. Further, using Theorem \ref{prop.20220607--03}, it sends non-zero objects
to non-zero ones. Since the mapping $R[x^\pm]\to \lau Rx$ is flat, the functor 
is exact and the standard criterion to verify {\it faithfulness} is assured.

To end the proof, we establish {\it fullness} by following the idea behind the proof of Proposition \ref{prop.20230730}. Let then $M$
 and $N$ be objects of $\mcrs(R[x^\pm]/R)$ and let 
\[
\bb r_{M,N}:\mm{Hom}_\mc( M , N )\aro\mm{Hom}_\mc(\bb r M ,\bb r N )\]
be the map which we want to show is surjective. 
We proceed in several steps. Let $\ph\in\mm{Hom}_\mc( M ,N)$. 

{\it First case:  $M$ and $N$ are $R$-flat.} Surjectivity of   $\bb r_{M,N}$ was verified in Theorem \ref{thm.20220623--01}.

{\it Second case: $M$ is  $R$-flat.}
Let $q:N'\to N$ be a surjection with $N'\in\mcrs^\circ(R[x^\pm]/R)$, see Corollary \ref{18.09.2023--1}.
We then construct the commutative diagram with exact rows
\[
\xymatrix{
0\ar[r]& \g M'' \ar[r]^U\ar[d]^\sim&  \g M' \ar[r]^{P} \ar@{}[rd]|\square\ar[d]_{\psi}&\bb rM\ar[r]\ar[d]^\varphi &0\\ 
0\ar[r]& \bb r N''\ar[r]_{\bb rv}&\bb rN' \ar[r]_{\bb r q}&\bb r N\ar[r]&0,}
\]
where the rightmost square is cartesian. (We slightly abuse notation and denote by $\bb r$ the map between sets of morphisms if no confusion is likely.)
It follows that $(0:t)_{\g M'}=0$, so that $\g M'=\bb r M'$  for some $M'\in\mcrs^\circ( R[x^\pm]/R)$ and $\ps=\bb r (\Ps)$, and $U=\bb ru$, by the preceding case.
 It is also true that $P=\bb rp$, so that we conclude that $\ph$ belongs to the image of $\bb r_{M,N}$. 

{\it Most general case.} Let $p:M'\to M$ be a surjection in $\mcrs(R[x^\pm]/R)$
with $M'\in\mcrs^\circ(R[x^\pm]/R)$ (cf. Corollary \ref{18.09.2023--1}). From the first case, there exists an arrow $\Ps:M'\to N$ such that 
\[
\xymatrix{
\bb rM'\ar[dr]_{\bb r\Ps}\ar[r]^{\bb rp}&\bb rM\ar[d]^\ph
\\
&\bb rN}
\]
commutes. Now, let $\iota:M''\to M'$ be the kernel of $p$ and note that $\bb r( \Ps)\bb r(\iota)=0$ and hence $\Ps\iota=0$. Let $\Ph:M\to N$ be the unique arrow 
rendering commutative the following diagram:
\[
\xymatrix{
M'\ar[r]^\Ps\ar[dr]_p&N
\\
&M\ar[u]_\Ph.
}
\]Because $\bb rp$ is an epimorphism, we conclude that $\bb r\Ph=\ph$. 
\end{proof}


\begin{center} 
\bf Acknowledgment
\end{center}
We would like to thank the anonymous referee for his/her careful reading, pointing
out a mistake in a previous version, and 
constructive suggestions leading to a significant improvement of our work.
\bibliographystyle{alpha}

\end{document}